\documentclass[a4paper,11pt]{article}
\usepackage{amsfonts}
\usepackage{amsmath}
\usepackage{amsbsy}
\usepackage{amsxtra}
\usepackage{latexsym}
\usepackage{amssymb}
\usepackage{url}
\usepackage{enumerate}
\usepackage{amscd}
\usepackage[active]{srcltx}
\usepackage{cite}
\newtheorem{theorem}{Theorem}
\newtheorem{lemma}{Lemma}

\newtheorem{proposition}{Proposition}
\newtheorem{remark}{Remark}

\newtheorem{example}{Example}
\newtheorem{problem}{Problem}

\def\min{\operatorname{Minimize}}
\def\const{\operatorname{subject~to~}}

\DeclareMathOperator{\Tol}{Tol}

\DeclareMathOperator{\diag}{diag}
\DeclareMathOperator{\sgn}{sgn}

\newcommand{\lng}{\langle}
\newcommand{\rng}{\rangle}

\newcommand{\R}{\mathbb R}

\newenvironment{proof}{{\noindent\bf Proof.}}{\hfill$\Box$\\}
\newenvironment{proof3}{{\noindent\bf Proof of Theorem 3.}}{\hfill$\Box$\\}

\begin{document}

\title{A semi-smooth Newton method for solving  convex quadratic   programming problem under   simplicial cone constraint 
\thanks{{\it 1991 A M S Subject Classification.} Primary 90C33;
Secondary 15A48, {\it Key words and phrases.} Metric projection onto simplicial cones}}

\author{ J. G. Barrios \thanks{IME/UFG, Campus II- Caixa Postal 131, 
Goi\^ania, GO, 74001-970, Brazil (e-mail:{\tt
numeroj@gmail.com}).  The author was supported in part by CAPES.}  
\and
 O. P. Ferreira\thanks{IME/UFG, Campus II- Caixa Postal 131, 
Goi\^ania, GO, 74001-970, Brazil (e-mail:{\tt
orizon@ufg.br}).  The author was supported in part by
FAPEG, CNPq Grants 4471815/2012-8,  305158/2014-7 and PRONEX--Optimization(FAPERJ/CNPq).}  
\and
S. Z. N\'emeth \thanks{School of Mathematics, The University of Birmingham, The Watson Building, 
Edgbaston, Birmingham B15 2TT, United Kingdom
(e-mail:{\tt nemeths@for.mat.bham.ac.uk}). The author was supported in part by 
the Hungarian Research Grant OTKA 60480.} }

\maketitle

\begin{abstract}
In this paper the simplicial cone constrained convex quadratic programming problem is studied.  The optimality conditions of this problem consist
in a linear complementarity problem. This fact, under a  suitable condition,  leads to an equivalence between the simplicial cone constrained 
convex quadratic programming problem and the one of finding the unique solution of a nonsmooth system of equations. It is shown that a 
semi-smooth Newton method applied to this nonsmooth system  of equations  is always  well defined and under a mild assumption on the simplicial 
cone the method  generates a  sequence  that  converges linearly to its solution. Besides, we also show that  the generated 
sequence is   bounded  for any  starting point and a formula for any accumulation point of this sequence is presented. The presented numerical 
results suggest that this approach achieves accurate solutions to large problems in few iterations. 


\end{abstract}

%%%%%%%%%%%%%%%%%%%%%%%%%%%%%%%%%%%%%%%%%%%%%%%%%%%%%%%%%%%%%%
\section{Introduction}
%%%%%%%%%%%%%%%%%%%%%%%%%%%%%%%%%%%%%%%%%%%%%%%%%%%%%%%%%%%%%%
The  purpose  of  this  paper  is  to  motivate  and  describe a  new  approach  for solving a special class of  constrained convex  quadratic programming problems, namely,   simplicial cone constrained ones, by using the semi-smmooth Newton's method, and to present the results of some computational experiments designed to investigate its  practical viability. 

Simplicial cone constrained convex quadratic  programming  arises  as an important problem in its own right, it has  an important  subclass of    
positively constrained  convex quadratic programming, namely, the positively constrained least-squares problems, or equivalently the problem 
of   projecting the point   onto a  simplicial cone.  The interest in the subject of projection arises in several situations,   having a wide range of applications in pure and applied mathematics such as  Convex Analysis 
(see e.g. \cite{HiriartLemarecal1}),   Optimization (see  e.g. \cite{BuschkeBorwein96,censor07,censor01,Frick1997,scolnik08,ujvari2007projection}), Numerical Linear Algebra (see e.g. \cite{Stewart77}), Statistics  (see e.g. \cite{BerkMarcus96,Dykstra83,Xiaomi1998}), Computer Graphics (see e.g. \cite{Fol90} ) and  Ordered 
Vector Spaces (see e.g. \cite{AbbasNemeth2012,IsacNem86,IsacNem92,NemethNemeth2009,Nemeth20091,Nemeth2010-2}).   More specifically, the projection onto a polyhedral  cone,  which has as a special case the projection onto a simplicial one, is a problem of high impact on scientific community\footnote{see the popularity of the 
Wikimization page Projection on Polyhedral Cone at http://www.convexoptimization.com/wikimization/index.php/Special:Popularpages}.  The geometric  nature of this problem makes it particularly interesting and important  in many areas of science and technology such as   Statistics~(see e.g. \cite{Xiaomi1998}), 
Computation~(see e.g. \cite{Huynh1992}), Optimization (see  e.g.\cite{Morillas2005,ujvari2007projection}) and Ordered Vector Spaces (see e.g. \cite{NemethNemeth2009}). 

 The projection onto a  general simplicial cone is difficult and computationally expensive, this problem has been studied e.g. in \cite{AlSultanMurty1992,EkartNemethNemeth2009,Frick1997,MurtyFathi1982,NemethNemeth2009,ujvari2007projection}.  It is a special convex quadratic
program and its  KKT optimality conditions consists in  a linear complementarity problem (LCP) associated with it,  see e.g
\cite{Murty1988,MurtyFathi1982,ujvari2007projection}.  Therefore, the problem of projecting onto  simplicial cones can be solved by  active set
methods \cite{Bazaraa2006,LiuFathi2011,LiuFathi2012,Murty1988} or any algorithms for solving LCPs,  see e.g  \cite{Bazaraa2006,Murty1988} and
special methods based on its geometry,   see e.g \cite{MurtyFathi1982,Murty1988}. Other fashionable ways to solve this problem are based on
the classical von Neumann algorithm (see e.g. the Dykstra algorithm \cite{DeutschHundal1994,Dykstra83,Shusheng2000}). Nevertheless, these 
methods are also quite expensive (see the numerical results in \cite{Morillas2005} and the remark preceding section 6.3 in 
\cite{MingGuo-LiangHong-BinKaiWang2007}).

The KKT optimality conditions  of simplicial cone constrained convex quadratic  programming consist  in  a  linear complementarity problem. Under
a  suitable condition, this leads to an equivalence between the simplicial cone constrained convex quadratic  programming problem  and the one of
finding the unique solution of a nonsmooth system of equations.  It is worth pointing out that a similar equation has been studied by Mangasaria 
in \cite{Mangasarian2009}, which have used  the semi-smooth Newton method for solving that equation, namely, an absolute value equation.   
Following the idea of \cite{Mangasarian2009},  we apply the semi-smooth Newton's method, see \cite{LiSun93},   to find a unique solution of the  associated nonsmooth
system of equations, which generates the solution of the simplicial cone constrained convex quadratic programming. Under a mild assumption on the
simplicial cone we show that the method  generates a  sequence  that  converges linearly to the  solution of the  associated  system of
equations.   This new approach has potential advantages over existing methods. The main advantage appears to be the  ability to achieve accurate
solutions to large problems in relatively few iterations. The global and linear convergence properties partially explain this good behavior. Our
numerical results suggest that for a given problem class, the number of required iterations is almost unchanged.  The numerical results also indicate a remarkable  robustness   with respect  to the starting point. 
  
  The organization of the paper is as follows.  In Section~\ref{sec:int.1},  some notations and  basic results  used in the paper are presented.
  In the beginning of Section~\ref{sec: defbp} our main problem, the simplicial cone constrained convex quadratic programming problem, is 
  presented. In Section~\ref{sec:seaqp} a semi-smooth equation is presented whose any solution generates a solution of our convex quadratic 
  programming problem and an existence and uniqueness result of the solution for this semi-smooth equation is obtained. In 
  Section~\ref{sec:ssnm2}  we  state and prove a convergence  theorem on the semi-smooth Newton method for  finding the solution of the 
  semi-smooth equation associated to the simplicial cone constrained convex quadratic programming problem. In Section~\ref{sec:ctest} we present 
  some computational tests.
%%%%%%%%%%%%%%%%%%%%%%%%%%%%%%%%%%%%%%%%%
\subsection{Notations and auxiliary results} \label{sec:int.1}
%%%%%%%%%%%%%%%%%%%%%%%%%%%%%%%%%%%%%%%%%

In this subsection we fix the notations and presend some auxiliary results used throughout the paper.
Let $\R^n$ denote the $n$-dimensional Euclidean space and let 
$\lng\cdot,\cdot\rng$  be the canonical scalar product and $\|\cdot\|$ be the norm generated by it.    The  $i$-th component of a vector $x
\in\R^n$ is denoted by $x_i$ for every $i=1,\ldots,n$. Define the nonnegative orthant as 
$$
\R^n_+ :=\{x\in\R^n~:~x_i\ge 0\,,\,j=1,\ldots,n\}. 
$$
For $x\in \R^n$, $\sgn(x)$ will denote a vector with components equal to $1$, $0$ or $-1$ depending on whether the corresponding component of the
vector $x$ is positive, zero or negative.    If $a\in\R$ and $x\in\R^n$, then denote $a^+:=\max\{a,0\}$, $a^-:=\max\{-a,0\}$ and $x^+$, $x^-$ and
$|x|$  the vectors with  $i$-th component equal to $(x_ i)^+$, $(x_ i)^-$ and  $|x_i|$, respectively. From the definitions of $x^+$ and  $x^-$ it easy to conclude that 
\begin{equation} \label{eq;drp}
x=x^+- x^-, \qquad  x^+ \in \R^n_+, \qquad   x^- \in \R^n_+, \qquad \langle  x^+,   x^-\rangle=0,  \qquad \forall ~x\in\R^n. 
\end{equation}
\begin{remark} \label{re:ppo}
It is well know that the projection onto a convex set  is continuous and nonexpansive, see \cite{HiriartLemarecal1}.  Since  the projection of the point    \(x \in\R^n\) onto the  nonnegative orthant is \(x^+\), we conclude that   \(\|z^+-w^+\|\leq \|z-w\|\), for all \(z, w \in\R^n\). 
\end{remark}
The set of all $m \times n$ matrices with real entries is denoted by $\R^{m \times n}$ and    $\R^n\equiv \R^{n \times 1}$.  The matrix $I$
denotes the $n\times n$ identity matrix.  If $x\in \R^n$ then $\diag (x)$ will denote an  $n\times n$  diagonal matrix with $(i,i)$-th entry equal
to $x_i$, $i=1,\dots,n$. For an $M \in \R^{n\times n}$ consider the norm defined by  $\|M\|:=\max_{x\ne0}\{\|Mx\|~:~  x\in \R^{n}, ~\|x\|=1\}$. This definition implies
\begin{equation} \label{eq:np}
\|Mx\|\leq \|M\|\|x\|, \qquad \|LM\|\le\|L\|\|M\|, 
\end{equation}
for any  matrices $L, M \in \R^{n\times n}$.  
\begin{lemma}[Banach's Lemma] \label{lem:ban}
Let $E \in \R^{n\times n}$   and   $I$ the $n\times n$ identity matrix. If  $\|E\|<1$,  then $E-I$
is invertible and  $ \|(E-I)^{-1}\|\leq 1/\left(1-
\|E\|\right). $
\end{lemma}
We will call a closed set $K\subset\R^n$ with nonempty interior a \emph{cone} if the following conditions hold:
\begin{enumerate}
	\item $\lambda x+\mu y\in K$ for any $\lambda,\mu\ge0$ and $x,y\in K$,
	\item $x,-x\in K$ implies $x=0$.
\end{enumerate}
Let \(K \subset \R^n\) be a cone. The {\it dual cone} of \( K\) is the following set  
\[
 K^*:=\{ x\in \R^n \mid  \langle x, y \rangle\geq 0, \forall \, y\in K\}.
\]
The   \emph{simplicial cone} associated to a   nonsingular matrix  $A \in \R^{n\times n} $  is defined by
\begin{equation} \label{d:sc}
A\R^n_+:=\{Ax~:~ x\in\R^n_+\}, 
\end{equation}
The following result follows from the  definition of the dual of a cone. For a proof see for example \cite{AbbasNemeth2012}.
\begin{lemma}\label{le:lad}  
Let $A$ be an $n\times n$ nonsingular matrix. Then,
\[
(A\mathbb{R}^n_+)^*=(A^\top)^{-1}\mathbb{R}^n_+.
\]
\end{lemma}
We will   need the following  result, for a proof  combine Proposition~{2A.3}  with {Theorem~{2A.6} of \cite{DontRockBook2009}.
\begin{theorem} \label{th:pol}
Let   \(\varphi :\R^n \to \R\) be a differentiable  convex function and $K$ be a closed,  convex cone in \(\R^n \). Then \(\bar{x}\)  is a solution of the problem
\begin{align*}  
 & \min  ~ \varphi(x) \\  \notag
 & \const  x \in K, 
\end{align*}
if  and only if  \( \bar{x}\) satisfies   the following optimality conditions
\[
x \in K,  \qquad \nabla \varphi(x) \in K^*, \qquad      \left\langle \nabla \varphi(x), x \right\rangle =0.
\]
\end{theorem}
 We end this section with the basic contraction mapping principle, its proof can be found in of \cite{DontRockBook2009} (see Theorem 1A.3 page 15).
\begin{theorem}  [basic contraction mapping principle] \label{fixedpoint}
 Let ${\mathbb X}$ be a complete metric space with metric $\rho$ and let $\phi : {\mathbb X} \to  {\mathbb X}$.  Suppose that there exists $\lambda \in  [0,1)$ such that  $\rho(\phi(x), \phi(y)) \le \alpha \rho(x,y)$, for all $ x, y \in {\mathbb X}$. Then there exists an unique $x\in {\mathbb X}$  such that $\phi(x) = x$
\end{theorem}
%%%%%%%%%%%%%%%%%%%%%%%%%%%%%%%%%%%
\section{Quadratic   programming  under a  simplicial cone constraint} \label{sec: defbp}
%%%%%%%%%%%%%%%%%%%%%%%%%%%%%%%%%%%
In this section we will present  a  semi-smooth Newton method for solving a special   class  of quadratic   programming problems, namely,  quadratic  programming  problems  under a  simplicial cone constraint.  The statement of such a problem is:
\begin{problem}[{quadratic programming  problem under a  simplicial cone constraint}]  \label{prob:qpscc}
Let $Q \in \R^{n\times n}$  be  a symmetric positive definite matrix,  $b \in \R^n$, $ c \in \R$ and  $A  \in \R^{n\times n}$ a nonsingular
matrix.  Find  a  solution $u$ of the convex programming problem
\begin{align}  %\label{eq:napsc2}
 & \min   ~ \frac{1}{2}x^\top Qx+x^\top b +c \\  \notag
 & \const  x \in A\R^n_+.
\end{align}
\end{problem}
Let  us  present  an important particular case of Problem~\ref{prob:qpscc}. 
\begin{example}
 Given  $A \in \R^{n\times n}$    a nonsingular matrix  and  $z\in \R^n$.   The {\it projection $P_{A\R^n_+}(z)$ of the point $z$ onto the cone  \( A\R^n_+ \)} is defined by 
\[
P_{A\R^n_+}(z):=\mbox{argmin} \left\{ \frac{1}{2}\|z-y\|^2~:~ y\in A\R^n_+\right\}.
\]
From the definition of the simplical cone associated with  the  matrix   $A$ in \eqref{d:sc}, the problem of   projecting the point $z\in \R^n$
onto a  simplicial cone   $A\R^n_+$ may be  stated  as the following  positively constrained  quadratic programming problem 
\begin{align*}
 & \min  ~ \frac{1}{2}\|z-Ax\|^2, \\ 
 & \const  x \in \R^n_+.
\end{align*}
Hence, if $v\in \R^n$ is the   unique  solution of this problem then  we have   $P_{A\R^n_+}(z)=Av$.  The above problem is equivalent to the
following   nonegatively constrained  quadratic programming problem
\begin{align}  \label{eq:nap}
 & \min  ~ \frac{1}{2}x^\top \tilde{Q}x+x^\top \tilde{b} +\tilde{c} \\  \notag
 & \const  x \in \R^n_+, 
\end{align}
by taking $\tilde{Q}=A^\top A$,  $\tilde{b}=-A^\top z$ and $\tilde{c}=z^\top z /2$.  The  optimality condition for problem \eqref{eq:nap}  implies  that its solution  can be obtained by solving the following linear complementarity problem 
\begin{equation} \label{eq:lcp}
y - \tilde{Q}x= \tilde{b}, \qquad x \ge 0, \qquad    y \ge0, \qquad    \langle x, y \rangle=0.
\end{equation}
% where $M=A^\top A$, $q=- A^\top z$.%
It is  easy to establish  that corresponding to each  nonnegative quadratic problems  \eqref{eq:nap} and each  linear complementarity problems \eqref{eq:lcp} associated to  symmetric positive definite matrices, there are equivalent  problems of   projection onto  simplicial cones.   Therefore, the problem of projecting onto  simplicial cones can be solved by  active set
methods \cite{Bazaraa2006,LiuFathi2011,LiuFathi2012,Murty1988} or any algorithms for solving LCPs,  see e.g  \cite{Bazaraa2006,Murty1988} and
special methods based on its geometry,   see e.g \cite{MurtyFathi1982,Murty1988}. Other fashionable ways to solve this problem are based on
the classical von Neumann algorithm (see e.g. the Dykstra algorithm \cite{DeutschHundal1994,Dykstra83,Shusheng2000}). Nevertheless, these 
methods are also quite expensive (see the numerical results in \cite{Morillas2005} and the remark preceding section 6.3 in  \cite{MingGuo-LiangHong-BinKaiWang2007}). 
\end{example}
In the next section we will show that Problem~\ref{prob:qpscc}  can be solved by finding a solution of a special semi-smooth equation.
%%%%%%%%%%%%%%%%%%%%%%%%%%%%%%%%%%%%%%%%%%%%%%%%%%%%%%%%%%%%%%%%%%%%
\subsection{The semi-smooth equation associated to  quadratic  programming} \label{sec:seaqp}
%%%%%%%%%%%%%%%%%%%%%%%%%%%%%%%%%%%%%%%%%%%%%%%%%%%%%%%%%%%%%%%%%%%%
In this section we present a semi-smooth equation whose  any solution  generates a solution of Problem~\ref{prob:qpscc}.  
\begin{problem}[{\bf semi-smooth equation}]  \label{prob:ne}
Let $Q \in \R^{n\times n}$  be  a symmetric positive definite matrix,  $b \in \R^n$, $ c \in \R$ and  $A \in \R^{n\times n}$ a nonsingular matrix. Find  a  solution $u$ of the semi-smmoth  equation
\begin{equation}\label{equation}
		\left[A^\top QA-I\right]x^+ +x+A^\top b=0.
\end{equation} 
\end{problem}
Next we  apply   Theorem~\ref{th:pol}  for showing  that   a solution of    Problem~\ref{prob:ne}  generates a solution of Problem~\ref{prob:qpscc}.
\begin{proposition}  \label{pr:polscq}
If the vector \(u\)  is a solution of Problem~\ref{prob:ne}, then  \( Au^+\)   is a solution of Problem~\ref{prob:qpscc}.
\end{proposition}
\begin{proof}
Note that  from \eqref{eq;drp} we have $u^+-u=u^-$ for all \(u\in
\R^n\). Thus,  if \(u\in \R^n\) is a solution of Problem~\ref{prob:ne}, then
\[
  A^\top \left(QAu^+ + b\right)=u^-. 
\]
Since $A$ is a nonsingular matrix and  \(u^-\in \mathbb{R}^n_+\),  it follows from  the last equality that 
\[
QAu^+ + b =(A^\top)^{-1}u^-\in (A^\top)^{-1}\mathbb{R}^n_+.
 \]
Hence,  by using Lemma~\ref{le:lad}  and  $\left\langle   u^-,  u^+ \right\rangle=0$, the last inclusion easily  implies that 
\[
QAu^+ + b  \in (A\mathbb{R}^n_+)^*, \qquad \langle  QAu^+ + b ,  Au^+ \rangle =  0. 
\]
Therefore,  as \(Au^+ \in A\mathbb{R}^n_+$,  applying  Theorem~\ref{th:pol} with   $K=A\mathbb{R}^n_+$ and \(\varphi(x) = x^\top Qx/2+x^\top b +c\) ,  the desired result follows.
\end{proof} 
\begin{proposition}\label{pr:uniqq}  
Let \( \lambda \in \mathbb{R}\).       If \(\left\|A^\top Q A-I\right\|\leq \lambda<1 \)   then  Problem~\ref{prob:ne}  has a unique solution. 
\end{proposition}
  \begin{proof}  The Problem~\ref{prob:ne}  has a solution  if only if the function 
 \(
 \phi (x)=-\left[A^\top QA-I\right]x^+ -A^\top b, 
 \)
 has a fixed point.  From  the definition   of the function \( \phi\) and \eqref{eq:subgrad}, it follows that for all \(x, y \in \mathbb{R}^n\) we have 
\[
 \phi(x)- \phi(y)=\int_{0}^{1}  -\left[A^\top Q A-I\right]\mbox{diag}\left(\mbox{sgn}((y+t(x-y))^+)\right)(x-y) \mbox{d}t.
\]
Since \( \|\mbox{diag}(\mbox{sgn}((y+t(x-y))^+))\|<1\) and   \(\left\|A^\top Q A-I\right\|<\lambda<1 \) for all \(t\in [0,1]\),  the last equality implies that 
\[
\| \phi(x)- \phi(y)\|\leq \lambda \|x-y\|, \qquad \forall ~x, y \in \mathbb{R}^n.
\]
Hence \( \phi\) is a contraction. Therefore  applying Theorem~\ref{fixedpoint} with ${\mathbb X}= \mathbb{R}^n$ and \(\rho=\| . \|\) we conclude that \(\phi \) has  precisely a unique fixed point and consequently Problem~\ref{prob:ne}    has precisely a unique solution.
\end{proof}
%%%%%%%%%%%%%%%%%%%%%%%%%%
\subsection{Semi-smooth Newton method} \label{sec:ssnm2}
%%%%%%%%%%%%%%%%%%%%%%%%%%
In this section our goal is to state and prove a convergence  theorem on the semi-smooth Newton method for  finding the solution of
Problem~\ref{prob:ne}. We will first prove the well-definedness of the sequence generated by the semi-smooth Newton method. Then, under suitable
conditions,  the $Q$-linear convergence   will be established. We also give a condition for the Newton method to finish in a finite number of 
iterations.  Finally, we show that the semi-smooth sequence generated by the Newton method  is bounded and we give a formula for any accumulation 
point of it. The statement of the main theorem is:
\begin{theorem}\label{th:mrq}
Let $Q \in \R^{n\times n}$  be  a symmetric positive definite matrix,  $b \in \R^n$, $ c \in \R$ and  $A \in \R^{n\times n}$ a nonsingular
matrix.    Then, the sequences \(\{x_k\}\)  generated by the semi-smooth Newton Method for solving   Problem~\ref{prob:ne}, 
 \begin{equation} \label{eq:newtonc2}
	 x_{k+1}=-\left ( \left[A^\top Q A-I\right]\mbox{diag}(\mbox{sgn}(x_k^+)) +I\right )^{-1} A^\top b , 
\end{equation}
for \(k=0,1,\ldots \), is well defined  for  any  starting  point   \(x_0\in \mathbb{R}^n\).  Moreover, if 
\begin{equation} \label{eq:CC2}
\left\|A^\top Q A-I\right\|<1/2, \qquad   \forall ~ x\in \mathbb{R}^n, 
\end{equation}
 then  the sequence $ \{x_k\}$   converges $Q$-linearly  to  \(u\in \mathbb{R}^n\),  the unique solution  of  Problem~\ref{prob:ne}, as follows 
  \begin{equation} \label{eq:lconv3}
\|u-x_{k+1}\|\leq \frac{\|A^\top QA -I\|}{1- \|A^\top Q A -I\|} \|u-x_{k}\|,  \qquad k=0, 1, \ldots ,
\end{equation}
As a consequence,  \( Au^+\)   is the solution of the  Problem~\ref{prob:qpscc}.
\end{theorem}
Henceforward  we assume that all assumptions in Theorem~\ref{th:mrq} hold.  The {\it semi-smooth Newton method}, see \cite{LiSun93},  for solving the  Problem~\ref{prob:ne}, i.e.,  for  finding the zero of the function 
\begin{equation} \label{eq:fuc}
F(x):=\left[A^\top QA-I\right]x^+ +x+A^\top b. 
\end{equation}
with starting point   \(x_0\in \mathbb{R}^n\),  is formally  defined by 
\begin{equation} \label{eq:nmqc}
F(x_k)+ V_k\left(x_{k+1}-x_{k}\right)=0,  \qquad   V_k \in \partial F(x_k), \qquad k=0,1,\ldots,
\end{equation}
where  \( V_k\) is any subgradient in  \( \partial F(x_k)\)  the  Clarke generalized Jacobian of \(F\) at \(x_k\).  Note that   \(S(x)\),   a subgradien  in the    Clarke generalized Jacobian of the function $F$ at $x$  (see   Definition~ 2.6.1 on page 70 of  \cite{Clarke1990}),  is given by
\begin{equation} \label{eq:subgrad}
S(x):=\left[A^\top QA-I\right]\mbox{diag}(\mbox{sgn}(x^+)) +I\in  \partial F(x), \qquad x\in  \R^n. 
\end{equation} 
Since  
\(
\mbox{diag}(\mbox{sgn}(x^+))x=x^+ 
\)
for all \(x\in \R^n\),  taking into account \eqref{eq:fuc} and \eqref{eq:subgrad}, we conclude that $S(x)x=F(x)- A^\top b$.  Thus,      taking  $V_k=S(x_k)$,    equation  \eqref{eq:nmqc}   becomes   
\begin{equation}\label{eq:mss}
	S(x_k)x_{k+1}= -A^\top b,   \qquad k=0, 1,  \ldots , 
\end{equation}
which is an equivalente definition of semi-smooth Newton sequence for solving the semi-smooth Problem~\ref{prob:ne}, i.e., equation \eqref{eq:newtonc2},  which   formally  defines  a sequence  \(\{x_k\}\) with starting point \(x_0\in \R^n \). Hence,  the sequence  \(\{x_k\}\)  defined in   \eqref{eq:newtonc2}  will be called  {\it semi-smooth Newton sequence} for  finding the zero of the function \(F\)  defined in \eqref{eq:fuc},  or equivalently  for solving the semi-smooth Problem~\ref{prob:ne}.
\begin{lemma}\label{nonsingGC}
The matrix \(S(x)\)   defined in  \eqref{eq:subgrad}  is nonsingular for all \(x\in \mathbb{R}^n\). As a consequence,  the  semi-smooth Newton sequence \(\{x_k\}\)    is well-defined,  for any starting point \(x_0 \in \R^{n}\).
\end{lemma}
\begin{proof}
 Let \(x\in \mathbb{R}^n\). To simplify the notations let  \(D=\mbox{diag}(\mbox{sgn}(x^+))\).     Thus, the matrix in $S(x)$ becomes
\[
 \left[A^\top Q A-I\right]D +I. 
\]
 Let us suppose, by contradiction, that this matrix  is singular, i.e, there exists \(u\in \mathbb{R}^n\) such that 
\[
\left(\left[A^\top Q A-I\right]D +I\right)u = 0,  \qquad  u \neq 0.
\]
It is straightforward to see that the last formula is equivalent to
\begin{equation} \label{eq:nsq}
A^\top Q ADu = (D - I)u,  \qquad  u \neq 0.
\end{equation}
Since the  matrix \(Q\) is symmetric and positive definite,  there exists  a  nonsingular matrix \(L \in \R^{n\times n}\) such that \(Q=LL^\top\). Taking into account  that   $D^2 = D$ and \(Q=LL^\top\), the  equality in equation \eqref{eq:nsq}  easily implies that 
\[
\left\| L^\top ADu\right\|^2 = \left\langle DA^\top Q ADu, u  \right\rangle=   \left\langle   (D^2-D)u, u\right\rangle  = 0. 
\]
Thus we have  \( L^\top ADu = 0\).  As \(Q=LL^\top\) and  \( L^\top ADu = 0\),  equation  \eqref{eq:nsq} implies  that \((D -I)u = 0\), or equivalently, \(Du=u\). Hence 
\[
 L^\top Au =  L^\top ADu = 0,   \qquad  u \neq 0.
\]
But this contradicts the nonsingularity of \(A\), since $L$ is nonsingular. Therefore, the matrix \(S(x)\) is nonsingular for all \(x\in \mathbb{R}^n\)  and the first part of the lemma is proven.
	
The proof of the second part of the lemma is an immediate consequence of the definition of the semi-smooth Newton sequence \(\{x_k\}\) in
\eqref{eq:newtonc2}, the definition of \(S(x)\) in \eqref{eq:subgrad}, and the first part of the lemma.
\end{proof}
\begin{lemma}\label{nonsingq}
 Let  \(S(x)\) be   as defined in  \eqref{eq:subgrad}.  If \(\left\|A^\top QA-I\right\|<1 \)  for all \(x \in \mathbb{R}^n\) then   
\[
\| S(x)^{-1}\| \leq \frac{1}{1- \left\|A^\top Q A -I\right\|}, \qquad  \forall ~x \in \mathbb{R}^n.
\]
\end{lemma}
\begin{proof}  To simplify the notation take $S(x)= -(E-I)$, where the matrix $E$ is defined by 
\[
E= \left[I-A^\top Q A \right]\mbox{diag}(\mbox{sgn}(x^+)).
\]
Since the diagonal matrix  $\mbox{diag}(\mbox{sgn}(x^+))$  has   components equal to $1$ or $0$, the definition of $E$ and the assumption 
\( \|A^\top Q A -I\| <1 \)  implies that  
\[
\|E\|\leq \|A^\top Q A -I\| <1.
\]
Therefore,  as $S(x)= -(E-I)$, combining the last inequality with Lemma~\ref{lem:ban} and the definition of $E$,  the desired inequality follows. 
\end{proof}
\begin{lemma}\label{eq:ftcq}
Let  \(F\) be the  function defined in \eqref{eq:fuc} and  \(S(x)\) be the matrix  defined in \eqref{eq:subgrad}. Then the following
inequality holds:
\[
\left\| S(x) -S(y)\right \|  \leq  \|A^\top Q A -I\|, \qquad  \forall ~x, y \in \mathbb{R}^n.
\]
As a consequence, 
\[
 \left\|F(x)-F(y)-S(y)(x-y)\right\| \leq   \|A^\top Q A -I\| \|x-y\|, \qquad  \forall ~x, y \in \mathbb{R}^n.
\]
\end{lemma}
\begin{proof}
Let $x, y \in \mathbb{R}^n$.  The definition     in  \eqref{eq:subgrad}  implies that 
\[
 \left\|S(x) -S(y)\right\|=  \|A^\top Q A -I\| \left \|\mbox{diag}(\mbox{sgn}(x^+)) -  \mbox{diag}(\mbox{sgn}(y^+)) \right\| \leq  \|A^\top Q A -I\|, 
\]
which is the first inequality of the lemma.  For proving the second  inequality of the lemma,   note that     the definitions in \eqref{eq:fuc} and \eqref{eq:subgrad} imply 
\[
F(x)-F(y)-S(y)(x-y)=\int_{0}^{1}  \left[ S(y+t(x-y))-S(y)\right] (x-y) \mbox{d}t.
\]
Therefore, the result follows by taking the norm in both sides of the last equality and using  the first part of the lemma.  
\end{proof}

Finally, we are ready to prove  the  main result, namely, Theorem~\ref{th:mrq}. 
\medskip

\begin{proof3} The well-definedeness,  for any starting point \(x_0 \in \R^{n}\),  follows from Lemma~\ref{nonsingGC}.  Using  Proposition~\ref{pr:uniqq},  we conclude that under assumption \eqref{eq:CC2} Problem~\ref{prob:ne} has unique solution \(u\in  \R^{n}\).

Let  \(F\) be the  function defined in \eqref{eq:fuc} and  \(S(x)\) be the matrix  defined in \eqref{eq:subgrad}. Since \(u\in  \R^{n}\) is the solution of Problem~\ref{prob:ne}  we have \( F(u)=0\), which together  with definition of   $\{x_k\}$ in \eqref{eq:newtonc2} implies 
$$
u-x_{k+1}= -S(x_k)^{-1}\left[F(u) - F(x_k)- S(x_k)(u-x_{k})\right],  \qquad k=0, 1, \ldots .
$$
Using properties of the norm  in \eqref{eq:np},   last equality implies 
$$
\|u-x_{k+1}\|\leq  \|S(x_k)^{-1}\|  \left\|\left[F(u)- F(x_k)- S(x_k)(u-x_{k})\right]\right\|,  \qquad k=0, 1, \ldots .
$$
Combining Lemma~\ref{nonsingq} with the second part of Lemma~\ref{eq:ftcq}, we conclude from the last equality that 
\begin{equation} \label{eq:lconvq}
\|u-x_{k+1}\|\leq \frac{\|A^\top Q A -I\|}{1- \|A^\top Q A -I\|} \|u-x_{k}\|,  \qquad k=0, 1, \ldots .
\end{equation}
Since \( \|A^\top Q A -I\|<1/2 \),  we have \(\|A^\top Q A -I\|/(1- \|A^\top Q A-I\|)<1\). Therefore,   the inequality in \eqref{eq:lconvq} implies that \(\{x_k\}\) converges Q-linearly,  from any starting point,  to the 
solution \(u\) of Problem~\ref{prob:ne}. Hence the first part of the theorem is proven. 

Since \(u\in  \R^{n}\) is the solution of Problem~\ref{prob:ne}   the second part of the theorem follows  by using Proposition~\ref{pr:polscq}.
\end{proof3}

The next proposition gives a condition for the Newton iteration \eqref{eq:newtonc2} to finish in a finite
number of steps.
\begin{proposition} \label{pr:ft}
	If in \eqref{eq:newtonc2} it happens that $\mbox{sgn}(x_{k+1}^+)=\mbox{sgn}(x_k^+)$,   then $x_{k+1}$ is a solution of Problem~\ref{prob:ne} and \( Ax_{k+1}^+\)   is the solution of the  Problem~\ref{prob:qpscc}.
\end{proposition}
\begin{proof}
	If $\mbox{sgn}(x_{k+1}^+)=\mbox{sgn}(x_k^+)$ in equation \eqref{eq:newtonc2}, then it becomes 
	\begin{equation}\label{ek}
		\left\{\left[A^\top Q A-I\right]\mbox{diag}(\mbox{sgn}(x_{k+1}^+))+I\right\}x_{k+1}=-A^\top b.
	\end{equation}
	Since $\mbox{diag}(\mbox{sgn}(x_{k+1}^+))x_{k+1}=x_{k+1}^+$, the last equality  yields
	\[
	\left[A^\top QA-I\right]x_{k+1}^++x_{k+1}=-A^\top b, 
	\] 
	which implies that  $x_{k+1}$ is a  solution of   Problem~\ref{prob:ne} and,  by using  Proposition~\ref{pr:polscq},
	%Proposition~\ref{pr:polscGC}  with    \(\varphi= \psi\) 
	it follows  that \( Ax_{k+1}^+\)   is the solution of the  
	Problem~\ref{prob:qpscc}.
\end{proof}

The next proposition shows that the semi-smooth Newton sequence $\{x_k\}$, defined in \eqref{eq:newtonc2},  is bounded and gives a formula for any accumulation point of it, without assuming condition \eqref{eq:CC2}.
\begin{proposition} \label{pr:bounded}
         The semi-smooth Newton sequence $\{x_k\}$,  defined in \eqref{eq:newtonc2},   is bounded from any starting point. Moreover, for each accumulation point $\bar x$ of  $\{x_k\}$ there exists 
	 an $\hat x \in \mathbb{R}^n$ such that
\begin{equation} \label{eq:sap}
 \left(\left(A^\top Q A-I\right) \mbox{diag}(\mbox{sgn}(\hat{x}^+))+I\right){\bar x}=-A^\top b.
\end{equation}
	 In particular, if   \(\mbox{sgn}(\bar x^+)=\mbox{sgn}(\hat x^+)\),   then \( \bar x\)  is a solution of Problem~\ref{prob:ne} and \( A\bar x^+\)   is the solution of Problem~\ref{prob:qpscc}.

\end{proposition}
\begin{proof}
 	Suppose to the contrary that $\{x_k\}$ is unbounded. Note that, as  there are only finitely many vectors 
	$\mbox{sgn}(x_{k}^+)$ with coordinates $0$ or $1$,  there exists a vector $\tilde{x}\in\mathbb{R}^m$ and a subsequence $\{x_{k_i}\}$ of $\{x_k\}$ such that 
	$$
	 \qquad\mbox{sgn}(x_{k_i}^+)\equiv\mbox{sgn}(\tilde{x}^+).
	$$
	Now, since $\{x_k\}$ is unbounded and  the unit sphere is compact, there exists a vector $v\in\mathbb{R}^m$ and a subsequence $\{x_{k_j}\}$ of $\{x_{k_i}\}$ such that 
	\begin{equation} \label{eq:sssp}
	\lim_{j\to\infty}{\|x_{k_j+1}\|}=\infty,  \qquad \lim_{j\to\infty}\frac{x_{k_j+1}}{\|x_{k_j+1}\|}=v\ne0.
	\end{equation}
	Therefore, as $\mbox{sgn}(x_{k_j}^+)= \mbox{sgn}(\tilde{x}^+)$ for all $j$, the definition of the semi-smooth Newton sequence $\{x_k\}$ in 
	\eqref{eq:mss} implies
	$$
	\left(\left(A^\top QA-I\right)\mbox{diag}(\mbox{sgn}(\tilde{x}^+))+I\right)\frac{x_{k_j+1}}{\|x_{k_j+1}\|}=-\frac{A^\top b}{\|x_{k_j+1}\|},\qquad j=0,1,2,\ldots.
	$$
	By tending with $j$ to infinity in the above equality and by taking into account \eqref{eq:sssp}, it follows that 
	$$
	\left(\left(A^\top Q A-I\right)\mbox{diag}(\mbox{sgn}(\tilde{x}^+))+I\right)v=0, 
	$$
	which contradicts the first part of the Lemma \ref{nonsingGC} since $v\ne0$. Therefore, the sequence  $\{x_k\}$ is bounded, which proves the first part of the 
	proposition.   
	 
	For proving the second part of the proposition, let  $\bar x$ be an accumulation point of the sequence  $\{x_k\}$.  Then, since there are only finitely many vectors
	$\mbox{sgn}(x_{k}^+)$ with coordinates  $0$ or $1$, there exists a vector $\hat{x}\in \mathbb{R}^m$ and a subsequence $\{x_{k_j}\}$ of $\{x_k\}$ such that 
	$$
	\lim_{j\to\infty}{x_{k_j+1}}=\bar x, \qquad   \mbox{sgn}(x_{k_j}^+)\equiv  \mbox{sgn}(\hat{x}^+),
	$$
	Since  $\mbox{sgn}(x_{k_j}^+)=\mbox{sgn}(\hat{x}^+)$  for all $j$, the definition of the semi-smooth Newton sequence $\{x_k\}$ in
	\eqref{eq:newtonc2} implies	
	$$ 
	\left(\left(A^\top QA-I\right) \mbox{diag}(\mbox{sgn}(\hat{x}^+))+I\right)x_{k_j+1}=-A^\top b,\qquad j=0,1,2,\ldots.
	$$
	Taking the limit in the last equality as $k_j$ goes to $\infty$ ,  the second part of the proposition follows.
	
	Finally,   for proving last part of the proposition, use  the assumption  \(\mbox{sgn}(\bar x^+)=\mbox{sgn}(\hat x^+)\)  and  \eqref{eq:sap} to obtain
	\[
	\left(\left(A^\top Q A-I\right) \mbox{diag}(\mbox{sgn}(\bar x^+))+I\right){\bar x}=-A^\top b.
	\]
Therefore, taking into account that \(\mbox{diag}(\mbox{sgn}(\bar x^+))\bar x=\bar x^+\) it is easy to conclude from the above equality that   
\(\bar x^+\) is a  solution of   Problem~\ref{prob:ne} and,  by using   Proposition~\ref{pr:polscq}, we obtain   that \( A\bar x^+\)   is the solution of the  Problem~\ref{prob:qpscc}, which conclude the proof of the proposition.
\end{proof}
%%%%%%%%%%%%%%%%%%%%%%%%%
%%%%%%%%%%%%%%%%%%%%%%%%%
\section{Computational results} \label{sec:ctest}
\label{sec:computationalresults}

In this section we test our semi-smooth Newton method (\ref{eq:newtonc2}) to find solutions on generated random instances of 
Problem~\ref{prob:ne}. We present two types of experiments. In one of them, we guarantee that for each test problem the hypotheses given in 
Theorem~\ref{th:mrq} are satisfied and in the other they are not.

 All programs were implemented in MATLAB Version 7.11 64-bit and run on a $3.40 GHz$ Intel Core $i5-4670$ with $8.0GB$ of RAM. All MATLAB codes and generated data of this paper are available in \url{http://orizon.mat.ufg.br/pages/34449-publications}. 

All experiments are based on the following general considerations:
\begin{itemize}
    \item In order to accurately measure the method's runtime for a problem, each one of the test problems  was solved $10$ times and the runtime data collected. Then, we defined the corresponding  {\it method's  runtime for a problem} as the median of these measurements.

\item Let $\Tol X\in \R_+$ be a relative bound, we consider that the method converged to the solution and stopped the execution when, for some $k$, the condition 
$$
%\|u-x_{k}\|<\Tol X*(1+\|u\|), 
\|u-x_{k}\|<\Tol X(1+\|u\|), 
$$ is satisfied. If the previous stopping criteria are not met within $100$ iterations, we declare that the method did not converge.
\end{itemize}

\subsection{When the hypotheses of Theorem~\ref{th:mrq} are satisfied}
In this experiment, we studied the behavior of the method on sets of $100$ randomly generated test problems of dimension $n=2000,3000,4000,5000$,
respectively. Furthermore, we analyzed the influence of the initial point in the convergence of the method on $1000$ randomly generated test problems of dimension
$n = 100$. %We guarantee that for%
For each test problem in this experiment the hypotheses given in the Theorem~\ref{th:mrq} are satisfied, generating each of them as follows:
%and each of them was generated as follows:
\begin{enumerate}
     \item To construct the matrices $A,Q\in \R^{n\times n}$ satisfying the assumption (\ref{eq:CC2}) in Theorem~\ref{th:mrq}, we first chose a
	     random number $\beta$ from the standard uniform distribution on the open interval $(0,1/2)$. Secondly, we compute the symmetric
	     positive definite matrix $Q=B^TB$, where $B$ is a generated $n\times n$ real nonsingular matrix containing random values drawn 
	     from the uniform distribution on the interval $[-10^6,10^6]$. Then, we compute the matrices $S,V$ and $D$, respectively, from the 
	     singular value decomposition of a generated $n\times n$ real nonsingular matrix containing random values drawn
from the uniform distribution on the interval $[-10^6,10^6]$. Finally, we compute the matrix $A$ from the system of linear equations 
$$
BA =S~\mbox{sqrt}\left(I+\frac{\beta}{\sigma} V\right)~D, 
$$
were  $\sigma$ is the largest singular value of $V$  and  $\mbox{sqrt}(I+\frac{\beta}{\sigma} V)$ is the square root of the  diagonal matrix $I+\frac{\beta}{\sigma} V$. 
\item   We have chosen the solution $u\in \R^{n}$ containing random values drawn from the uniform distribution on the interval $[-10^6,10^6]$ and
	then we have computed $b\in \R^{n}$ from equation (\ref{equation}). 
\item Finally we have chosen a  starting point $x_0\in \R^{n}$ containing random values drawn from the uniform distribution on the interval $[-10^6,10^6]$.
\end{enumerate}

In accordance with the theoretical convergence of the method, ensured by Theorem~\ref{th:mrq}, the computational convergence is obtained in all cases.

The computational results to analyze the behavior of the method on sets of $100$ generated random test problems of different dimensions, are reported in Table~\ref{tab:example1}. From these, it can be noted that for the same dimension, to achieve higher accuracy, the method does not experience a significant increase in the number of iterations or in runtime. On the other hand, the increase in the dimension of the problems does not necessarily involve  an increase in the number of iterations to achieve the same accuracy, however, a larger runtime is consumed. A larger runtime consumption is associated with the fact that the semi-smooth Newton method~(\ref{eq:newtonc2}) requieres the solution of a linear system in each iteration, whose computational effort increases with the dimension of the problem. Another important aspect that can be checked in Table~\ref{tab:example1} is the ability of the method to converge  in about three iterations on average.

\begin{table}[htbp]
\centering
\renewcommand{\arraystretch}{1.2}
\begin{tabular}{|c|lll|rrr|}
\hline
\multicolumn{1}{|c|}{$n$}&\multicolumn{3}{c|}{Total Iterations}&\multicolumn{3}{c|}{Total Time}\\\cline{1-7}
 $2000$& 284&  299&  300&  334.243&   351.823&   352.922\\
 $3000$& 279&  293&  295& 1064.158&  1117.909&  1124.941\\
 $4000$& 281&  303&  303& 2481.145&  2676.010&  2674.550\\
 $5000$& 283&  303&  305& 4927.101&  5261.154&  5142.072\\
\hline\hline
\multicolumn{1}{|c|}{$\Tol X$}&\multicolumn{1}{l}{$10^{-6}$}&\multicolumn{1}{l}{$10^{-8}$}&\multicolumn{1}{l|}{$10^{-10}$}&\multicolumn{1}{c}{$10^{-6}$}&\multicolumn{1}{c}{$10^{-8}$}&\multicolumn{1}{c|}{$10^{-10}$}\\
\hline
\end{tabular}
\caption{Total overall iterations and total time in seconds, performed and consumed, respectively by the semi-smooth Newton method (\ref{eq:newtonc2})  to solve the $100$ test problems  of each dimension for different accuracies.}
\label{tab:example1}
\end{table} 

In order to study the influence of the initial point in the convergence of the method,  we have generated $1000$ test problems of dimension 
$n = 100$ and we have associated to each of them $1000$ generated initial points. We have solved each problem with each of the $1000$ 
corresponding initial points. Then, we have computed the standard deviation ($\mbox{STD}$) $\overline{d}_i$ and the mean value (\mbox{MEAN})
$\overline{m}_i$ of the number of iterations performed by the method to solve the problem $i$ taking each one of the $1000$ initial points.
Finally we have computed the mean of all $\overline{d}_i$  and the mean of all $\overline{m}_i,\;\;i=1,...,1000$. All cases converged, indicating
robustness of the method with respect to the starting point. The results are shown in Table~\ref{tab:example2}. The standard deviation of  the
number of iterations performed by the method to solve the problem $i$ with  the $1000$ initial points  gives us an idea of the influence of the
initial point in the number of iterations performed by the method in each problem. The reported means of these  standard deviation values give 
us an idea of the influence of the initial point in the number  of iterations performed by the method in all the problems in general. The results
in the table show that on average the number of iterations performed by our method to find the solution for a problem varies only very slightly 
with the chosen starting point. Again we see that the average number of iterations performed is less than three.

\begin{table}[htbp]
\centering
\renewcommand{\arraystretch}{1.2}
\begin{tabular}{ccc}
\hline
Tol X&\multicolumn{1}{c}{$\mbox{MEAN}\left(\{\overline{d}_i\}_{i=1,...,1000}\right)$}&\multicolumn{1}{c}{$\mbox{MEAN}\left(\{\overline{m}_i\}_{i=1,...,1000}\right)$}\\\hline\hline
$10^{-6}$&0.241& 2.337\\
$10^{-8}$&0.249&  2.348\\
$10^{-10}$&0.249&  2.348\\\hline
\end{tabular}
\caption{Influence of the initial point in the convergence of the semi-smooth Newton method (\ref{eq:newtonc2}) on a total of  $1000$ test problems of dimension $n=100$ each of them with $1000$ generated initial points for different accuracies.}
\label{tab:example2}
\end{table} 

\subsection{When the hypotheses of Theorem~\ref{th:mrq} are not satisfied} 
In this experiment, we studied the behavior of the method on $1000$ test problems of dimension $n=100$, where the hypotheses given in the Theorem~\ref{th:mrq} are not all satisfied.

In this case, the test problems were built almost as in the previous experiment. The only difference was in the construction of the matrices
$A,Q\in \R^{n\times n}$ not satisfying the assumption (\ref{eq:CC2}) of Theorem~\ref{th:mrq}. Namely, we chose the random number $\beta$ from the
standard uniform distribution on the  interval $[lb,ub)$, where $\frac{1}{2}\leq lb<ub$. Then, $\|A^{T}QA-I\|=\beta$.

According to the obtained numerical results, we can conjecture that our method converges to a much broader class of problems, not satisfying the
hypotheses of Theorem~\ref{th:mrq}. However we detected that convergence with high accuracy to the solution largely depends on the magnitude of
the value of the norm in condition (\ref{eq:CC2}). This idea can be observed inspecting Table \ref{tab:example3}. As the magnitude of the value
of the norm in (\ref{eq:CC2}) increases, the number of problems for which the method converges  decreases, and decreases the number of
problems for which the method converges to the solution with greater accuracy. It can be also seen that for the same value of the norm in (\ref{eq:CC2}), the number of problems for which the method converges with greater precision reduces. This phenomenon, of course, is not associated to the convergence of the method for a specific problem, but, rather, there is an optimum accuracy achievable due to the accumulated errors. Small tolerances do not ensure obtaining accurate results. It can be the case that convergence is overlooked and unnecessary iterations are performed. It is important to note in the table that, even when the hypothesis is unfulfilled, the method converges for these problems in an average of less than seven iterations, which means an increase of approximately four iterations with respect of the previous experiments in which the hypotheses were fulfilled.

\begin{table}[htbp]
\centering
\renewcommand{\arraystretch}{1.2}
\begin{tabular}{|c|rrr|rrr|}
\hline
\multicolumn{1}{|c|}{$\beta\in[lb,ub)$}&\multicolumn{3}{c|}{Solved Problems}&\multicolumn{3}{c|}{Iterations}\\\cline{1-7}
 $[0.5,10^3)$ &1000&  1000&  994& 5.813&  5.813&  5.813\\
 $[10^3,10^4)$&1000&  1000&  966& 6.318&  6.318&  6.316\\
 $[10^4,10^5)$& 1000&   995&  539& 6.389&  6.389&  6.455\\
 $[10^5,10^6)$&1000&   964&    3& 6.436&  6.438&  6\\
 $[10^6,10^7)$& 995&   547&    0& 6.467&  6.497&  -\\
 $[10^7,10^8)$& 960&    3 &    0& 6.436&  6.667&  -\\\hline\hline
\multicolumn{1}{|c|}{Tol X}&\multicolumn{1}{c|}{$10^{-6}$}&\multicolumn{1}{c|}{$10^{-8}$}&\multicolumn{1}{c|}{$10^{-10}$}&\multicolumn{1}{c|}{$10^{-6}$}&\multicolumn{1}{c|}{$10^{-8}$}&\multicolumn{1}{c|}{$10^{-10}$}\\
\hline
\end{tabular}
\caption{Number of problems solved by the semi-smooth Newton method (\ref{eq:newtonc2}) on a total of  $1000$ test problems of dimension $n=100$ of each condition ($lb\leq\|A^{T}QA-I\|<ub$) for different accuracies, and the mean number of iterations performed by the semi-smooth Newton method (\ref{eq:newtonc2})  to solve one problem in each case.}
\label{tab:example3}
\end{table} 

\section{Conclusions} \label{sec:conclusions}
In this paper we studied  a special class of convex quadratic programs, namely, simplicial cone constrained convex quadratic programming
problems, which,  via  its optimality conditions, is reduced to  finding the unique solution of a  nonsmooth system of equations.  Our main 
result shows that, under a mild assumption on the simplicial cone, we can  apply  a semi-smooth Newton method for finding a unique solution of  
the  obtained  associated nonsmooth system of equations and  that  the generated sequence converges linearly to the  solution   for  any starting
point.  It would be interesting to see whether the used technique can be applied for solving  more general convex programs.    

Since the optimality condition of a simplicial cone constrained convex quadratic programming problem consists in a certain type of  linear
complementarity problem, which is equivalent to the problem of finding the unique solution of a nonsmooth system of equations, another  
interesting problem to address is to compare our semi-smooth Newton method   with   active set methods  
\cite{Bazaraa2006,LiuFathi2011,LiuFathi2012,Murty1988}. 

This paper is a continuation of  \cite{FerreiraNemeth2014}, where we studied the problem of projection onto a simplicial cone by using 
a semi-smooth Newton method. We expect that the results of this paper become a further step towards solving general convex optimization problems. We foresee further progress in this topic in the nearby future.
%%%%%%%%%%%%%%%%%%%%%%%%%%%%%%%%%%%%%%%%%%%%%%%
%%%%%%%%%%%%%%%%%%%%%%%%%%%%%%%%%%%%%%%%%%%%%%%%%%%%%%%%%%


\begin{thebibliography}{10}

\bibitem{AbbasNemeth2012}
M.~Abbas and S.~Z. N{\'e}meth.
\newblock Solving nonlinear complementarity problems by isotonicity of the
  metric projection.
\newblock {\em J. Math. Anal. Appl.}, 386(2):882--893, 2012.

\bibitem{AlSultanMurty1992}
K.~S. Al-Sultan and K.~G. Murty.
\newblock Exterior point algorithms for nearest points and convex quadratic
  programs.
\newblock {\em Math. Programming}, 57(2, Ser. B):145--161, 1992.

\bibitem{BuschkeBorwein96}
H.~H. Bauschke and J.~M. Borwein.
\newblock On projection algorithms for solving convex feasibility problems.
\newblock {\em SIAM Rev.}, 38(3):367--426, 1996.

\bibitem{Bazaraa2006}
M.~S. Bazaraa, H.~D. Sherali, and C.~M. Shetty.
\newblock {\em Nonlinear programming}.
\newblock Wiley-Interscience [John Wiley \& Sons], Hoboken, NJ, third edition,
  2006.
\newblock Theory and algorithms.

\bibitem{BerkMarcus96}
R.~Berk and R.~Marcus.
\newblock Dual cones, dual norms, and simultaneous inference for partially
  ordered means.
\newblock {\em J. Amer. Statist. Assoc.}, 91(433):318--328, 1996.

\bibitem{censor07}
Y.~Censor, T.~Elfving, G.~T. Herman, and T.~Nikazad.
\newblock On diagonally relaxed orthogonal projection methods.
\newblock {\em SIAM J. Sci. Comput.}, 30(1):473--504, 2007/08.

\bibitem{censor01}
Y.~Censor, D.~Gordon, and R.~Gordon.
\newblock Component averaging: an efficient iterative parallel algorithm for
  large and sparse unstructured problems.
\newblock {\em Parallel Comput.}, 27(6):777--808, 2001.

\bibitem{Clarke1990}
F.~H. Clarke.
\newblock {\em Optimization and nonsmooth analysis}, volume~5 of {\em Classics
  in Applied Mathematics}.
\newblock Society for Industrial and Applied Mathematics (SIAM), Philadelphia,
  PA, second edition, 1990.

\bibitem{DeutschHundal1994}
F.~Deutsch and H.~Hundal.
\newblock The rate of convergence of {D}ykstra's cyclic projections algorithm:
  the polyhedral case.
\newblock {\em Numer. Funct. Anal. Optim.}, 15(5-6):537--565, 1994.

\bibitem{DontRockBook2009}
A.~L. Dontchev and R.~T. Rockafellar.
\newblock {\em Implicit functions and solution mappings}.
\newblock Springer Monographs in Mathematics. Springer, Dordrecht, 2009.
\newblock A view from variational analysis.

\bibitem{Dykstra83}
R.~L. Dykstra.
\newblock An algorithm for restricted least squares regression.
\newblock {\em J. Amer. Statist. Assoc.}, 78(384):837--842, 1983.

\bibitem{EkartNemethNemeth2009}
A.~Ek\'art, A.~B. N\'emeth, and S.~Z. N\'emeth.
\newblock Rapid heuristic projection on simplicial cones, 2010.

\bibitem{FerreiraNemeth2014}
O.~Ferreira and S.~N\'{e}meth.
\newblock Projection onto simplicial cones by a semi-smooth newton method.
\newblock {\em Optimization Letters}, pages 1--11, 2014.

\bibitem{Fol90}
J.~D. Foley, A.~van Dam, S.~K. Feiner, and J.~F. Hughes.
\newblock {\em Computer Graphics: Principles and Practice}.
\newblock Addison-Wesley systems programming series, 1990.

\bibitem{Frick1997}
H.~Frick.
\newblock Computing projections into cones generated by a matrix.
\newblock {\em Biometrical J.}, 39(8):975--987, 1997.

\bibitem{HiriartLemarecal1}
J.-B. Hiriart-Urruty and C.~Lemar{\'e}chal.
\newblock {\em Convex analysis and minimization algorithms: Fundamentals. {I}},
  volume 305 of {\em Grundlehren der Mathematischen Wissenschaften [Fundamental
  Principles of Mathematical Sciences]}.
\newblock Springer-Verlag, Berlin, 1993.

\bibitem{Xiaomi1998}
X.~Hu.
\newblock An exact algorithm for projection onto a polyhedral cone.
\newblock {\em Aust. N. Z. J. Stat.}, 40(2):165--170, 1998.

\bibitem{Huynh1992}
T.~Huynh, C.~Lassez, and J.-L. Lassez.
\newblock Practical issues on the projection of polyhedral sets.
\newblock {\em Ann. Math. Artificial Intelligence}, 6(4):295--315, 1992.
\newblock Artificial intelligence and mathematics, II.

\bibitem{IsacNem86}
G.~Isac and A.~B. N{\'e}meth.
\newblock Monotonicity of metric projections onto positive cones of ordered
  {E}uclidean spaces.
\newblock {\em Arch. Math. (Basel)}, 46(6):568--576, 1986.

\bibitem{IsacNem92}
G.~Isac and A.~B. N{\'e}meth.
\newblock Isotone projection cones in {E}uclidean spaces.
\newblock {\em Ann. Sci. Math. Qu\'ebec}, 16(1):35--52, 1992.

\bibitem{LiuFathi2011}
Z.~Liu and Y.~Fathi.
\newblock An active index algorithm for the nearest point problem in a
  polyhedral cone.
\newblock {\em Comput. Optim. Appl.}, 49(3):435--456, 2011.

\bibitem{LiuFathi2012}
Z.~Liu and Y.~Fathi.
\newblock The nearest point problem in a polyhedral set and its extensions.
\newblock {\em Comput. Optim. Appl.}, 53(1):115--130, 2012.

\bibitem{Mangasarian2009}
O.~L. Mangasarian.
\newblock A generalized {N}ewton method for absolute value equations.
\newblock {\em Optim. Lett.}, 3(1):101--108, 2009.

\bibitem{Morillas2005}
P.~M. Morillas.
\newblock Dykstra's algorithm with strategies for projecting onto certain
  polyhedral cones.
\newblock {\em Appl. Math. Comput.}, 167(1):635--649, 2005.

\bibitem{Murty1988}
K.~G. Murty.
\newblock {\em Linear complementarity, linear and nonlinear programming},
  volume~3 of {\em Sigma Series in Applied Mathematics}.
\newblock Heldermann Verlag, Berlin, 1988.

\bibitem{MurtyFathi1982}
K.~G. Murty and Y.~Fathi.
\newblock A critical index algorithm for nearest point problems on simplicial
  cones.
\newblock {\em Math. Programming}, 23(2):206--215, 1982.

\bibitem{NemethNemeth2009}
A.~B. N\'emeth and S.~Z. N\'emeth.
\newblock How to project onto an isotone projection cone.
\newblock {\em Linear Algebra Appl.}, 433(1):41--51, 2010.

\bibitem{Nemeth20091}
S.~Z. N\'emeth.
\newblock Characterization of latticial cones in {H}ilbert spaces by
  isotonicity and generalized infimum.
\newblock {\em Acta Math. Hungar.}, 127(4):376--390, 2010.

\bibitem{Nemeth2010-2}
S.~Z. N\'emeth.
\newblock Isotone retraction cones in {H}ilbert spaces.
\newblock {\em Nonlinear Anal.}, 73(2):495--499, 2010.

\bibitem{LiSun93}
L.~Q. Qi and J.~Sun.
\newblock A nonsmooth version of {N}ewton's method.
\newblock {\em Math. Programming}, 58(3, Ser. A):353--367, 1993.

\bibitem{scolnik08}
H.~D. Scolnik, N.~Echebest, M.~T. Guardarucci, and M.~C. Vacchino.
\newblock Incomplete oblique projections for solving large inconsistent linear
  systems.
\newblock {\em Math. Program.}, 111(1-2, Ser. B):273--300, 2008.

\bibitem{Stewart77}
G.~W. Stewart.
\newblock On the perturbation of pseudo-inverses, projections and linear least
  squares problems.
\newblock {\em SIAM Rev.}, 19(4):634--662, 1977.

\bibitem{MingGuo-LiangHong-BinKaiWang2007}
M.~Tan, G.-L. Tian, H.-B. Fang, and K.~W. Ng.
\newblock A fast {EM} algorithm for quadratic optimization subject to convex
  constraints.
\newblock {\em Statist. Sinica}, 17(3):945--964, 2007.

\bibitem{ujvari2007projection}
M.~Ujv{\'a}ri.
\newblock On the projection onto a finitely generated cone, 2007.

\bibitem{Shusheng2000}
S.~Xu.
\newblock Estimation of the convergence rate of {D}ykstra's cyclic projections
  algorithm in polyhedral case.
\newblock {\em Acta Math. Appl. Sinica (English Ser.)}, 16(2):217--220, 2000.

\end{thebibliography}
\end{document}